\documentclass[11pt]{amsart}
\usepackage{amsmath,amsthm,amssymb,verbatim,hyperref}
\usepackage{graphicx}
\renewcommand{\r}[1]{(\ref{#1})} 
 
\newcommand{\be}[1]{\begin{equation}\label{#1}} 
\newcommand{\ee}{\end{equation}}

\renewcommand{\l}{\lambda}

\newcommand{\R}{{\bf R}}
\newcommand{\C}{{\bf C}}

\DeclareMathOperator {\diag} {diag}

\DeclareMathOperator {\supp} {supp}

\DeclareMathOperator {\Id} {Id}

 \renewcommand{\i}{\mathrm{i}}
 \renewcommand{\d}{\mathrm{d}}

   \newtheorem{theorem}    {Theorem}       [section]

    \newtheorem{prop}       [theorem]       {Proposition}
    \newtheorem{example}    [theorem]       {Example}

    \theoremstyle{definition}

    \theoremstyle{definition}
    \newtheorem{remark} [theorem]    {Remark}

\title[Weyl asymptotics]{Weyl  asymptotics of the transmission eigenvalues for a constant index of refraction }

\author[Ha Pham]{Ha Pham}
\address{Department of Mathematics, Purdue University, West Lafayette, IN 47907}

\author[P. Stefanov]{Plamen Stefanov}
\address{Department of Mathematics, Purdue University, West Lafayette, IN 47907}
\thanks{First author partly supported by a NSF  Grant DMS-1301646}

\begin{document}
\begin{abstract}
We prove  Weyl type of asymptotic formulas for the real and the complex internal transmission eigenvalues when the domain is a ball and the index of refraction is constant. 
\end{abstract}
\maketitle
\section{Introduction}\label{sec_intro}

The purpose of this paper is to prove a Weyl type of asymptotic formula for the counting function of the real Interior Transmission Eigenvalues (ITEs) in the simplest case: the domain is a ball and the index of refraction is a constant. 
Let $\Omega \subset \R^n$ be an open bounded domain with smooth boundary. Let $m>0$ be a smooth function in $\bar\Omega$. 
 The Interior  Transmission problem is given  by the following system
\begin{equation}\label{IITP}
 \begin{cases} (-\Delta  -\lambda^2m) u = 0 \\ (-\Delta - \lambda^2)    v= 0  \\
u|_{\partial\Omega} = v|_{\partial\Omega}, \; \partial_\nu u|_{\partial\Omega} = \partial_\nu v|_{\partial\Omega}  .
 \end{cases}
 \end{equation}
 Any $\lambda \not=0$ for which there exist non-zero $u, v \in H^2(\Omega)$ satisfying \eqref{IITP} is called an Interior Transmission Eigenvalue (ITE). We call the corresponding pair $(u,v)$ an Interior Transmission Eigenpair. One could think of $\mu=\l^2$ as the eigenvalues; then the real $\l^2\not=0$ correspond to the positive $\mu$. 
It is clear by unique continuation that if one of $u$ or $v$ is identically zero, then the other one vanishes as well, therefore the definition is unambiguous. One can replace the first equation with a more general, a possibly anisotropic Helmholtz type of equation, see \r{LVan} below.

 The interior transmission problem was first introduced in the context of inverse scattering problem by Colton and Monk in \cite{Colton_Monk_1988} and by Kirsch in \cite{Kirsch_1986}. Specifically, if one can extend  $v$
 outside $\Omega$ to a solution of the Helmholtz equation   $V$
of the type $V(x)=\int_{|\omega|=1}e^{\i \l\omega\cdot x}g(\omega) \d\omega$, (a Herglotz wave function) with $g\in L^2(S^{n-1})$, then $V$ is an incident wave that does not get scattered by the inhomogeneous media whose index of refraction is $m(x)$ in $\Omega$ and $1$ outside. 
More precisely, $g$ is an $L^2$ function in the kernel of $S-I$, where $S$ is the idenity operator. 
On the other hand, knowledge of the discreteness of the set of ITEs is essential in inverse scattering for acoustic and magnetic wave. It was shown in \cite{Kirsch_Grinberg_2008}, \cite{Cakoni_Colton_2006} and \cite{Cakoni_etal_2011} that
ITEs correspond to those frequencies where the reconstruction algorithm 
 which uses linear sampling method breaks down.

 The discreteness of ITEs was shown by Colton, Kirsch and P\"aiv\"arinta in \cite{Colton_etal_1989} under certain conditions.  
 The existence of real ITEs was established first for radially symmetric $m$, see, e.g., \cite{Colton_Monk_1988, McLaughlin_Polyakov_1994}   first and then for general $m$ by Sylvester and P\"aiv\"arinta in \cite{Sylvester_Paivarinta_2008}. The question on infiniteness and discreteness were extended by \cite{Cakoni_etal_2010a} on showing that the set of ITEs is infinite and discrete, under the assumption that the contrast $m-1$ in the medium does not change sign and is bounded away from zero for $-\Delta$ on $\R^3$. In \cite{Sylvester_2012}, discreteness is shown only under the requirement $m\not=1$ on the boundary, and results for $L^\infty$ and complex valued $m$ are presented as well, see also \cite{Lakshtanov_Vainberg_2013}. Robbiano \cite{Robbiano_2013} extended this to a larger class of complex-valued but smooth $m$. 
 
 Since the ITE problem is not self-adjoint, complex eigenvalues may exist. In \cite{Cakoni_etal_2010a}, it was shown the existence of complex TEs with some assumption on $m$. For infiniteness of the set of complex ITEs, see also the results in \cite{Hitrik_etal_2010}, \cite{Hitrik_etal_2011b} and  \cite{Leung_Colton_2012}. In \cite{Hitrik_etal_2010} and \cite{Hitrik_etal_2011b}, the authors extended the transmission problem to elliptic differential operators.

When the index of refraction $m$ is radially symmetric, ITEs completely determines $m$, as shown in \cite{Cakoni_etal_2010c} and \cite{McLaughlin_Polyakov_1994}.

  There are various recent results on asymptotic  distribution of the ITEs. It was shown in  \cite{Hitrik_etal_2011a}  that almost all ITEs are confined to a parabolic neighborhood of the positive real axis. Robiano in  \cite{Robbiano_2013} gave the following upper  bound
 $$N(r) \leq C r^{n + 2},\quad  r >1 .$$
 In \cite{Dimassi_Petkov_2013}, Dimassi and Petkov proved a sharp upper bound for the counting function of the complex ITEs (counted with their geometric multiplicities, see the discussion below)  in a small sector, namely
$$N(\theta, r) := \sharp \left\{ \lambda \in \mathbb{C} : \lambda\ \text{is an ITE}, \lvert \lambda \rvert \leq r, |\arg \lambda| \leq \theta \right\},$$ of the type 
\[
N(\theta,r)   \leq   \frac{\omega_n }{(2 \pi)^n} r^{n/2} \int_{\Omega} ( 1 +  m(x)^{n/2}) dx   +  O(r^{n - 1- \epsilon}/2),     \quad r \geq r_0(\theta, \epsilon),
\]
compare with  \r{LV} below. 
 In \cite{Lakshtanov_Vainberg_2013a}, Lakshtanov and Vainberg obtained a lower bound of the counting function $N(r)$ of the real ITEs; 
 for some $\delta > 0$
\be{LVa}
N(r) \geq \dfrac{\omega_n}{(2\pi)^n}r^{n}  \int_{\Omega} \left| 1 - 
m^{n/2}\right| 
\, \d x+ O(r^{n-\delta}), \  r\rightarrow \infty, \ \delta > 0.
\ee
They consider a more general case, see \r{LVan}, involving an anisotropic second order operator. 
Here $\omega_n$ is the volume of the unit ball in $\R^n$.  
 An $r^n/C$ lower estimates on the counting function for the real ITEs was also obtained  by Serov and Sylvester in \cite{Serov_Sylvester_2012}. 

The interior anisotropic transmission  problem  
\be{LVan}
\begin{cases} 
 -\Delta u(x) - \lambda^2 u(x) = 0,  & x \in \Omega,\\
 -\nabla A(x) \nabla v - \lambda^2 m(x) v = 0,  &x \in \Omega, \\
u(x) - v(x) = 0, &x\in \partial \Omega,\\
\tfrac{\partial u}{\partial \nu} - \nu\cdot A(x)\nabla v = 0,  & x\in \partial \Omega.
\end{cases}
\ee
has been studied as well. 
Here $A(x) = (a_{i,j}(x))$ is a real smooth symmetric elliptic matrix different from the identity matrix. In \cite{Lakshtanov_Vainberg_2012a, Lakshtanov_Vainberg_2012b, Lakshtanov_Vainberg_2012c, Lakshtanov_Vainberg_2013}, Lakshtanov and Vainberg 
established the discreteness and existence of the ITEs for the anisotropic case. Under some assumptions on $A(x)$,  in \cite{Lakshtanov_Vainberg_2012a}, they also stated  that the counting function
$$N(r) := \sharp \{ \lambda \in \mathbb{C} : \lambda \ \text{is an ITE}, \lvert \lambda \leq r\}$$
has the asymptotic
\be{LV}
N(r) \sim \dfrac{\omega_n}{(2\pi)^n} r^{n} \int_{\Omega} \Big[ 1 + \dfrac{ m^{n/2}(x)}{\det A(x)} \Big] \, dx , \ \ r>1.
\ee 
The assumptions on $A$ however exclude the case $A=\Id$. Such assumptions are needed to guarantee that the problem, microlocally restricted to the boundary, is elliptic in semiclassical sense (called there ``parameter elliptic'').

A Weyl type of asymptotic formula for radially symmetric $m$ corresponding to the momentum $l=0$ (an analogue of our 1D result below but with $m$ a function) was established in 
 \cite{McLaughlin_Polyakov_1994}. Such a formula in the 1D case is implicit in \cite{Sylvester_2013_1D}.

Before start counting the ITEs, we need to define a multiplicity of an ITE first. 
In the literature cited above, different notions are used, often implicit. 
For example, one can write the problem as a spectral problem for a certain fourth order elliptic differential operator, see, e.g., \cite{Sylvester_Paivarinta_2008} and study the corresponding null spaces. In \cite{Dimassi_Petkov_2013} and in \cite{Robbiano_2013}, see also \cite{Sylvester_2012}, the problem is viewed as a 
spectral one for the operator $P=\diag(-m^{-1}\Delta,-\Delta)$ with boundary conditions as in \r{IITP}.  
Then the ITEs are the eigenvalues of that non-selfadjoint operator, and there are might be generalized eigenvectors. The multiplicity is then the dimension of the generalized eigenspace, i.e., the rank of the residue of the resolvent. 
 One can also view the problem as a problem of finding the null-space of the difference $\text{DN}_m(\l)-\text{DN}(\l)$ of the Dirichlet-to-Neumann (DN) maps \cite{Lakshtanov_Vainberg_2013a}; or of the difference of  the Neumann-to-Dirichlet maps, when $\text{DN}_m(\l)$ and $\text{DN}(\l)$ have a common pole (which is exactly when $\l$ is not a simple ITE in the case we study). Then the null space consists of functions on the boundary (which can be related to interior solutions, of course). This formulation includes the spectral parameter in the operators in a non-linear way; and the implicit choice of the multiplicity then is the dimension of the null-space, which  does not include the generalized eigenvectors above. 

In our main theorem, we count the (real) ITEs with their geometric multiplicities defined as the dimension of the span of the eigenpairs $\{u,v\}$ corresponding to that ITE. With this convention, we set
\[
N(r) = \#\left\{\l;\;  0<\l\le r;\; \text{$\l$ is an ITE}             \right\}.
\]
Note that this excludes the contribution of  possible generalized eigenvectors of $P$ above.
 Our computations show that  gene\-ralized eigenvectors do exist, if $\l$ is not simple. The geometric multiplicity is always 1 in the 1D case, see Section~\ref{sec_2.2}; and equal to the dimension $\mu(l)$  of the spherical harmonics eigenspace for the corresponding momentum $l$ in the $n$-dimensional case. We define the algebraic multiplicity of an ITE $\l_0$ as the order of $\l_0$ a zero of the determinant $F_{\nu(l)}(\l)$, see \r{F}, of the system reflecting the boundary conditions when projecting to  a fixed spherical harmonics eigenspace. We observe the following interesting fact  in our case:  the  algebraic multiplicity  
is always either $1$ or $3$ (multiplied by $\mu(l)$ if we view  $F_\nu(\l)$  as a $F_\nu(\l)\Id_{\mathbf{C}^\mu}$).  In the 1D model case, it is  $1$ for all ITEs if and only if $\sqrt{m}$ ($\not=1$) is rational, see Proposition~\ref{pr1}. In the higher dimensional case, this depends on how $\sqrt m$ relates to the zeros of the Bessel functions $J_\nu$, see section~\ref{sec_nD}.

\begin{theorem}\label{thm1} 
Let $0<m\not=1$ be constant, and let $\Omega\subset\R^n$, $n\ge2$ be the unit ball. Then 
\[
\begin{split}
N(r)&= |N_1(r)-N_\gamma(r)|+O(r^{n-1})\\
  &= (2\pi)^{-n}\omega_n^2\left|1-m^{n/2}\right|r^n + O(r^{n-1}). 
\end{split}
\]
\end{theorem}
Here, $\omega_n$ is the volume of the unit ball in $\R^n$. Note that the factor $|1-m^{n/2}|$ is the same one which  appears in the known lower bounds of the real ITEs.  

 One motivation for studying the geometric, rather than the algebraic multiplicity is that the generalized eigenvectors do not show up when going back to the original motivation for studying the ITEs in the first place: can we tell $m$ from $1$ from boundary or scattering data. On the other hand, a higher algebraic multiplicity or closeby complex ITEs would matter if we are interested how sensitive the boundary or the scattering data are when $\l$ is close to $\l_0$. We refer also to Section~\ref{sec_n_mult} and Section~\ref{sec4} for more details. 

In the next section, we study a model problem: scattering on a half-line in one dimension. We  prove an asymptotic formula for the geometric and the  algebraic multiplicities as well and show that the leading term of the latter is an everywhere discontinuous function of $\gamma$. The real eigenvalues counted with their geometric multiplicities have the asymptotic in Theorem~\ref{thm1}. The  complex ones have a counting function with $|1-m^{n/2}|$ replaced by $1+m^{n/2}$ with $n=1$. This is the same factor that appears in the known upper bounds. 

\textbf{Acknowledgments.} The authors thank Vesselin Petkov for the helpful discussions about the multiplicity of the ITEs. 

\section{Transmission eigenvalue and Eigenspaces --- The half line case}\label{sec1D}
 
\subsection{Analysis of the eigenvalues} We analyze here a model problem: scattering on a half-line.  
In fact, the ITEs we get here are the same as in the 3D case when the angular  momentum $l$ is zero, see next section.  This is the main reason we study a half-line instead of  the whole line, as in \cite{Sylvester_2013_1D}.   
We view this as a model problem only; more complete results of this type for variable $m(x)$ can be found in  \cite{McLaughlin_Polyakov_1994}, \cite{ColtonPS_07}, and in \cite{Leung_Colton_2012}.  Sylvester  \cite{Sylvester_2013_1D} studied the same problem  on the whole line with $m$ constant and got precise results about the distribution of the ITEs in that case.

We set  $\gamma=\sqrt m>0$, which is assumed to be constant.

The 1D case we study  is  given by the system
\begin{equation}\label{tranHalfline}
 \begin{cases} -u'' -\lambda^2 u = 0 \\ -v''- \lambda^2 \gamma^2  v= 0  \\
 u(0)=v(0)=0\\
 u(1)=v(1);\ u'(1) = v'(1)
 \end{cases}.\end{equation}

Any $\l\not=0$ for which a non-trivial pair $(u,v)$ solving that system exists, is an ITE. 
For the purpose of this definition, $\gamma$ can be a function.

Then  $u$ and $v$ have the form:
 $$u = a \sin \lambda x,\quad v= b\sin \gamma \lambda x. $$

 The boundary  condition at $x=1$ gives
 \[
\begin{split} u(1) = v(1) \quad &\Longrightarrow \quad a \sin \lambda = b \sin \gamma \lambda , \\
 u' (1) = v'(1)\quad  &\Longrightarrow \quad a\lambda  \cos \lambda   = b \lambda \gamma \cos \gamma \lambda, \end{split}
\]
 which is written as, for $\lambda \neq 0$,
\be{system} 
M(\l) \begin{pmatrix} a\\ b\end{pmatrix}:= 
\begin{pmatrix}
 \sin \lambda  &- \sin \gamma \lambda  \\ 
 \cos \lambda  &- \gamma \cos \lambda \gamma 
 \end{pmatrix} \begin{pmatrix} a \\ b\end{pmatrix} = 0.
\ee
The above system has non trivial solution if the determinant of the matrix is zero, which gives the following  condition:
 \begin{equation}\label{trancon0}
 F(\lambda):= \gamma \sin \lambda   \cos \lambda \gamma  - \sin \lambda \gamma  \cos \lambda =0 .
 \end{equation}
This implies the following. 
   \begin{prop}
The (possibly complex) ITEs are the zeros of $F(\l)$ away from $\l=0$. 
   \end{prop}

It is easy to compute the following derivatives:
\be{F0}
\begin{split}
F'(\lambda ) &=  
 (1-\gamma^2) \sin \lambda  \sin (\gamma \lambda ),\\
F''(\lambda ) & = (1-\gamma^2)\left(  \cos \lambda  \sin (\gamma \lambda )
 +  \gamma \sin \lambda  \cos (\gamma \lambda )\right),\\
F'''(\lambda ) 
 & = (1-\gamma^2) \Big[ 2\gamma \cos \lambda  \cos \gamma \lambda  - (1+\gamma^2) \sin \lambda  \sin \gamma \lambda \Big].
  \end{split}
\ee
  
  \begin{prop}\label{pr1}
Let $\gamma(1-\gamma^2)\not=0$. Then all (possibly complex) roots of  $F$ have multiplicity one or three. 
If $\gamma$ is irrational, then $F(\lambda )$  has  single roots only. If $\gamma$ is rational, then  it has infinitely many   roots of multiplicity three. Moreover, $\l_0$ is a zero of multiplicity three if and only if $\sin\l_0=\sin(\gamma\l_0)=0$ and then necessarily $\l_0$ is real. 
  \end{prop}
  \begin{proof}
It is easy to see that   $F(\lambda_0) = F'(\lambda_0) = 0 $ if and only if
 $$   \sin(\lambda_0 ) = 0, \quad  \sin(\gamma \lambda_0) = 0 .$$
Then for any double root $\lambda_0$, 
$$F''(\lambda_0) = F'''(\lambda_0)=0$$
but
  $$F'''(\lambda_0) = \pm 2 \gamma (1-\gamma^2)\neq 0 .$$
Therefore, the multiplicity is either one or three.

Assume now that the multiplicity is three. Then (for $\l_0\not=0$)
  $$\begin{cases} F(\lambda_0) = 0\\
  F'(\lambda_0) = 0 \end{cases}
  \Leftrightarrow \begin{cases} \sin \lambda_0 = 0 \\ 
  \sin \gamma \lambda_0= 0 \end{cases}
 \Leftrightarrow \begin{cases} \lambda_0 =    k \pi, &k \in \mathbb{Z} \setminus \{0\} \\
  \gamma \lambda_0 =  l \pi, & l \in \mathbb{Z}\setminus \{0\} \end{cases}$$
  $$\Rightarrow \gamma =l/k \in \mathbb{Q} .$$
Therefore, multiplicity three implies that $\gamma$ is rational. On the other hand, if $\gamma=p/q$, then it follows from above that $kq\pi$, $k$ integer,  are all roots of multiplicity three. 
  \end{proof}

\begin{example}
When $\gamma=2$,  $F(\lambda) = -2\sin^3\lambda$ which has zeros of multiplicity three only. On the other hand, for $\gamma=3$, $F(\lambda)= -8\cos \lambda\sin^3\lambda$, which has both simple zeros and zeros  of multiplicity three. 

Then the counting function of the ITEs counted with their algebraic multiplicities (see the discussion below)  has to be multiplied by $3$ when $\gamma=2$; and by $2$, when $\gamma=3$ (modulo $O(1)$ in the latter case).  
\end{example}

\subsection{Analysis of the eigenspaces and multiplicities}\label{sec_2.2}
 Let $\gamma$ be irrational and let $\lambda_0$ be a simple ITE. Then $(\sin\lambda_0,-\sin\gamma\l_0)\not=(0,0)$, and by \r{system}, $a=\sin(\gamma\l_0)$, $b=\sin\l_0$, modulo a non-zero multiplicative constant. Then the eigenpair is given by
$$
u= \sin(\gamma\l_0)\sin(\l_0 x), \quad v = \sin\l_0\sin(\gamma\l_0 x).
$$

Assume now that $\lambda_0$ is a triple zero of $F$. Then $\sin\l_0= \sin(\gamma\l_0)=0$ but the matrix  $A$ in \r{system} still has rank one. Then, up to a multiplication by a non-zero constant,
$$
u= \gamma\cos(\gamma\l_0)\sin(\l_0 x), \quad v = \cos\l_0\sin(\gamma\l_0 x).
$$
Note that we do not get a space of larger dimension, even though the root is triple. On the other hand, in this case, $A^{-1}(\l)= (\l-\l_0)^{-3}B(\l)$, with $B(\l_0)\not=0$ (having rank one), i.e., the Laurent expansion of $A^{-1}(\l)$ at $\l=\l_0$ has its most singular order $-3$. One can compute that, when $\gamma=2$, the residue is of  rank one. Therefore, the candidate for the algebraic multiplicity, 3, is not the rank of the residue (which cannot be larger than 2 anyway), but it corresponds to the order of the most singular term. Therefore, if we accept the matrix $A(\l)$ to be our main object, we never get geometric multiplicity three in the ways we tried, while the number three is a natural choice of the algebraic multiplicity of $\l_0$.

Another point of view might be to go back to the motivation to study the ITE in the first place. They are $\l$'s at which we cannot tell $n(x)$ from $1$ by looking at Cauchy data on $\partial\Omega$. The Cauchy data for $n=1$, and $n=\gamma$, respectively, are given by 
\[
\{C(\sin\l,\cos\l)\}, \; \{C(\sin(\gamma\l),\gamma\cos(\gamma\l))\}.
\]
For $C=1$, the first vector is unit, and the second one is unit in another (equivalent, dependent on $\gamma$ only) norm. Therefore, the modulus of the determinant $F(\l)$ can be taken as a measure of how close those two 1D spaces are. 

This discussion suggests the following definition. The geometric multiplicity of an ITE is the dimension of the eigenspace of $A(\l)$ (always one). The algebraic one is the multiplicity as a root of $\det A(\l)$ (one or three).  We will count the ITEs below with their geometric multiplicities (i.e., once). This is consistent with the choice we made in the Introduction. 

\subsection{Counting the ITEs in the 1D case}
Let $N_\R^\text{\rm geom}(r)$ be the number of the real ITEs not exceeding $r$, counted once, i.e., with their geometric multiplicities. Let $N_1(r)$ be the number of the Dirichlet eigenvalues $\l^2$ with $\l\in (0,r]$ of $-d^2/dx^2$, and 
let $N_\gamma(r)$ be related in the same way to $-{\gamma^{-2}}d^2/dx^2$.  Clearly, $N_1(r) =  r/\pi+O(1)$, $N_\gamma (r)= \gamma r/\pi+O(1)$.

\begin{theorem}\label{thm_1D}
Let $0<\gamma\not=1$. Then 
\[\begin{split}
N_\R^\text{\rm geom}(r) &= |N_1(r)-N_\gamma(r)|+O(1)\\
   &= |1-\gamma|\frac{r}\pi +O(1). 
\end{split}
\]
\end{theorem}

\begin{proof}
Let $0<\gamma<1$ first. Assume that $\gamma$ is irrational. Then $\sin\l$ and $\sin(\gamma\l)$ do not have common zeros and the ITEs can be characterized as not only as the zeros of $F$ but also as the zeros of the differences $\text{DN}= \text{DN}_\gamma-\text{DN}_1=F/(\sin(\gamma\l)\sin\l)$ of the DN maps
\[
\text{DN}(\l) = \gamma\cot(\gamma\l)- \cot\l.
\]
Since by \r{F0}, $\text{DN}= (1-\gamma^2)F/F'$, at the zeros of $F$, we get $\text{DN}'=1-\gamma^2>0$. This is the crucial observation which makes the counting possible, compare with Proposition~\ref{prop_2}.  Therefore, $\text{DN}$ has at most one zero on any interval, on which it is continuous. In terms of the branches of the graphs (the graph between two consecutive poles) of $\gamma\cot(\gamma\l)$ and $\cot\l$, this means if two such branches intersect, then the former has a greater slope than the latter (and they are both negative); and therefore, they cannot intersect more than once. On the other hand, comparing their asymptotes at the poles, we see that a branch of $\cot\l$ is intersected by   a branch of (the slower varying) $\gamma\cot(\gamma\l)$ if and only if the interval of definition  of the former is contained in the (larger) interval of definition of the latter. Therefore, the number of branches of $\cot\l$ which do not contribute  (exactly one) zero is equal to the number of intervals $(k\pi,(k+1)\pi)$, $k=1,2,\dots$, which contain a zero $m\pi/\gamma$, $m$ integer, of $\sin(\gamma\l)$. Any such interval contains exactly one such zero because $\gamma<1$. 

Let $r>0$. We apply the arguments above for all intervals between zeros of $\gamma\cot(\gamma\l)$ in $[0,r]$. In the partial interval, which contains $r$, we have $O(1)$ ITEs. 
Therefore, up to an $O(1)$ error,  the number of ITEs is that of zeros of $\sin\l$ minus those of $\sin(\gamma\l)$.

Let $\gamma=p/q$ be rational next. Then each common pole of  $\cot\l$ and $\gamma\cot(\gamma\l)$ corresponds to common zeros of $\sin\l$ and $\sin(\gamma\l)$. By the analysis above, any such pole is an ITE. If $\gamma=p/q<1$ is irreducible, then the zeros of $\sin(\gamma\l)$ are $kq\pi/p$, $k$ integer. Those of them which are zeros of $\sin\l$ as well are given by $k=mp$, $m$ integer. Consider the interval between two such consecutive zeros, $I:= ((m-1)q\pi,mq\pi]$. As above, in $I$, the branch with domain $I_l:=((l-1)\pi,l\pi)$, $l=(m-1)q+1,\dots mq$, of the faster oscillating $\cot\l$ are not intersected by some of those of $\cot(\gamma\l)$ if and only if there is a pole of $\cot(\gamma\l)$ in $\bar I_l$; i.e., if some $kq\pi/p$ belongs to $[l\pi,(l+1)\pi]$, see Figure~\ref{pic1}. 
That pole can be an endpoint of $\bar I_l$ only for the most-left and the most-right intervals $I_l$; and there there no intersections of the two graphs. For  the rest of the $I_l$ intervals, and there are $q-2$ of them, the  pole of $\cot(\gamma\l)$ is interior for $I_l$, and therefore,  an interior point for $I$ as well. The number of such $kq\pi/p$ in the interior of $I$ is $p-1$; therefore we have $(q-2)-(p-1)$ ITEs in the interior of $I$; and therefore, $q-p$ ITEs in $I$
since we include the right endpoint in $I$ but not the left one. Then $N(mq\pi) = m(q-p)=mq(1-\gamma)$. This easily implies $N(r) = r(1-\gamma)/\pi+O(1)$. 

In section~\ref{sec_3.4} we use  a lightly modified counting argument which we could have applied to that case as well. 

When $\gamma>1$, one can rescale to show $N^\gamma(r/\gamma) = N^{1/\gamma}(r)$, where we temporarily denote by $N^\gamma$ the counting function $N_\R^\text{\rm geom}$ of the ITEs related to $\gamma$ (not to be confused with $N_\gamma$). Then we use what we proved above. 
\end{proof} 

\begin{figure}[!ht] 
  \centering
  \includegraphics[scale=1,keepaspectratio=true]{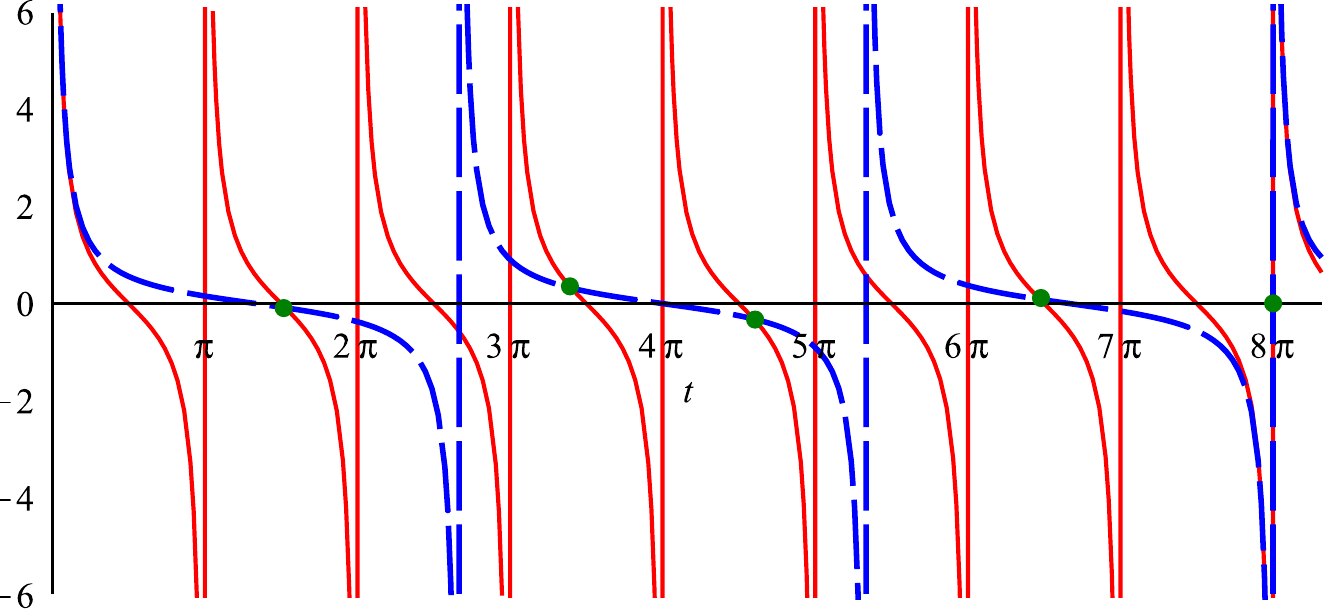}
\caption{The graphs of $\gamma\cot(\gamma \l)$ and $\cot(\l)$ on $(0,8\pi]$ with $\gamma=3/8$. Here, $p=3$, $q=8$. The thick  dots represent the ITEs. The ITE $8\pi$ has  algebraic multiplicity $3$, the other three are simple.}
  \label{pic1}
\end{figure}

Let $N^{\rm alg}_\R(r)$ be the number of the real ITEs not exceeding $r$ counted with their algebraic multiplicities ($1$ or $3$). The we get a different asymptotic.
 
\begin{theorem}\label{thm_1D2}
Let $\gamma=\sqrt{n}$ be a positive constant. 
If $\gamma=\sqrt{m}\not=1$ is irrational, then $N^{\rm alg}_\R(r)=N^{\rm geom}_\R(r)$. If $\gamma=p/q$ is rational and $p/q$ is irreducible,  then 
\[ 
N^{\rm alg}_\R(r) = \left(\Big|1-\frac{p}{q}\Big|+\frac2q\right) r/\pi +O(1). 
\]
\end{theorem}

\begin{proof}
If $\gamma$ is irrational, then we just apply Proposition~\ref{pr1}. Let $0<\gamma<1$ be rational. 
In the counting argument above, we counted each common zero of $\sin(\gamma\l)$ and $\sin(\l)$ once; and those are exactly the triple roots of $F$. When we work with the algebraic multiplicities, we should add them two more times. Therefore, in each interval $I$ as in the proof above, we have $q-p+2$ ITEs instead of $q-p$. As above, we need to divide this by $\pi q$ to get the leading term. When $\gamma>1$, we use the rescaling argument above, or direct counting. 
\end{proof} 

\begin{remark}
In both cases, $\gamma<1$ and $\gamma>1$, we get
\[
N^{\rm alg}_\R(r)\le (1+\gamma)r/\pi+O(1), 
\]
which is consistent with the theorem below. 
We have equality if and only if $\gamma$  or $1/\gamma$ is an integer.   
We get the same  leading term as in \r{LV} which estimates the \textit{complex} ITEs, and is also the upper bound in \cite{Dimassi_Petkov_2013}, modulo the multiplicative factor $3\sqrt3$.
\end{remark}

\subsection{Complex ITEs in 1D}

Let $N^{\rm alg}_\C(r)$ be the number of the complex ITEs in $\Re\l>0$ counted with their algebraic multiplicities ($1$ or $3$) of modulus $r$ or less. The fact that they are in a strip parallel to the real line has been proved in  \cite{Sylvester_2013_1D} and \cite{Leung_Colton_2012}, and the latter work provides an explicit bound for $C(\Gamma)$ below for the whole line case. 
 
\begin{theorem}\label{thm_1D3}
Let $\gamma=\sqrt{n}$ be a positive constant, different from $1$. Then  
\[ 
N^{\rm alg}_\C(r) = (1+\gamma) \frac{r}{\pi} +o(r). 
\]
Moreover, all ITEs are symmetric about the real line and included in the strip $|\Im\l|\le C(\gamma)$ for some $C(\gamma)>0$. 
\end{theorem}
\begin{proof}

Write
\[
4\i F(\l) =\gamma (e^{\i\gamma\l}+e^{-\i\gamma\l})(e^{\i\l}-e^{-\i\l})- (e^{\i\gamma\l}-e^{-\i\gamma\l})(e^{\i\l}+e^{-\i\l}).
\]
When $\l=\Re\l+\i\Im\l$, $\Im\ll0$, the leading term on the right is $(\gamma-1)e^{\i(\gamma+1)\l}$ with modulus $|\gamma-1|e^{(\gamma+1) |\Im\l|} $, and we get
\[
4|F(\l)|\ge |\gamma-1|e^{(\gamma+1) |\Im\l|} -C(\gamma) |\gamma-1|e^{|1-\gamma||\Im\l|}.
\]
Therefore, for $\Im\le -C'(\gamma)$, $F$ cannot vanish. Since $\overline{F(z)}=F(\bar z)$, this covers the $\Im\l>0$ case, as well. 

To prove the asymptotic formula, we apply a theorem by Titchmarsh \cite{Titchmarsh26} which was generalized to distributions, and successfully used by Zworski \cite{Zworski87} to prove a Weyl type of asymptotic of the resonances for the Schr\"odinger equation  in one dimension. If $ f$ is a distribution on $\R$ and $[a,b]$ is the smallest closed interval containing $\supp\hat f$, then for the counting function $N(r)$ of the zeros of $f$ in $\C$, we have
\[
N(r) = (b-a)r/\pi+o(r). 
\]
We can write $\hat F(\xi)$ as a linear combination of delta functions supported at $-\gamma-1$, $-\gamma+1$, $\gamma-1$, and $\gamma+1$, all with non-zero coefficients. Therefore, $a=-\gamma-1$, $b=\gamma+1$. We get twice the claimed asymptotic but this included the zeros in $\Re\l<0$ which are symmetric to those in $\Re\l>0$ because $F$ is odd. This completes the proof. 
\end{proof}

\begin{remark}
Combining Theorem~\ref{thm_1D2} and Theorem~\ref{thm_1D3}, we can estimate the asymptotic distribution of the non-real ITEs. Since it changes in a singular way when perturbing $\gamma$, this means that we can view the triple (almost real) ITEs  when $\gamma$ is rational as collapsed complex ITEs. In particular, we recover one of the result in \cite{Leung_Colton_2012}: we get existence of infinitely many complex eigenvalues when $\gamma$ or $1/\gamma$ is not an integer. We also get a linear lower $r/C$ bound on their counting function. 
\end{remark}

  \section{ITEs in higher dimensions}\label{sec_nD}

\subsection{Separation of variables} 
Denote by $Y_l^m$, $l=0,1,\dots$, $m=1,\dots \mu(l)$, an orthonormal set of spherical harmonics on $S^{n-1}$. They are the eigenfunctions of the Laplacian $\Delta_{S^{n-1}}$ on $S^{n-1}$. We have
$$
-\Delta_{S^{n-1}}Y_l^m = l(l+n-2)Y_l^m, \quad l=0,1,\dots; \;m=1,\dots, \mu(l),
$$
where, for each $l$, the multiplicity of the eigenvalue $l(l+n-2)$ is given by
\be{l}
\mu (l)= \frac{2l+n-2}{n-2}{l+n-3\choose n-3} = \frac{2l^{n-2}}{(n-2)!}\left(1+O(l^{-1})\right).
\ee
The functions 
\be{j}
j_\nu(\l) := \l^{1-n/2} J_{l+n/2-1}(\l), \quad \nu:=l+n/2-1
\ee
 are bounded  at $\l=0$; in fact, $J_l(\l)\sim c_l\l^l$, as $\l\to0$. 
Any solution $u$ of the Helmholtz equation $(-\Delta-\lambda^2)u=0$ near $0$ has the form
\be{u}
u(x) = \sum_{l=0}^\infty \sum_{m=1}^{\mu(l)} a_{lm} j_{l+n/2-1}(\l r) Y_l^m(\omega), 
\ee
where $x=r\omega$ and $r>0$, $|\omega|=1$ are polar coordinates. 
Similarly, any outgoing solution at $\infty$ has similar expansion, with $J_\nu$ replaced by $H^{(1)}_\nu$. Any solution $v$ of the  equation $(-\Delta-\lambda^2\gamma^2)u=0$ near $0$ has the form
$$
v(x) = \sum_{l=0}^\infty\sum_{m=1}^{\mu(l)} b_{lm} j_{l+n/2-1}(\gamma\l r) Y_l^m(\omega). 
$$
Assume that $u$ and $v$ are in $H^2(\Omega)$, where $\Omega$ is the unit ball. Then $u$ and $u_r$ restricted to $r=1$ are in $H^{3/2}$, and $H^{1/2}$, respectively and the series below converge. 

The boundary conditions in \r{IITP} imply
\be{F}
F_\nu(\l) := \gamma j_{\nu}(\l) j_{\nu}'(\gamma \l) - j_{\nu}(\gamma \l) j_{\nu}'(\l)=0 
\ee
for some $\nu$, which can be written also as 
  $$F_\nu(\l) = \gamma J_{\nu}(\l) J_{\nu}'(\gamma \l) - J_{\nu}(\gamma \l) J_{\nu}'(\l)=0. $$
Indeed, the Cauchy data for the unperturbed equation is 
\[
\left(\sum_{lm} a_{lm} j_{l+n/2-1}(\l ) Y_l^m(\omega), \; \l\sum_{lm} a_{lm} j'_{l+n/2-1}(\l ) Y_l^m(\omega)\right),
\]
and the the Cauchy data for the perturbed equation is
$$
\left(\sum_{lm} b_{lm} j_{l+n/2-1}(\gamma\l ) Y_l^m(\omega), \; \gamma\l\sum_{lm} b_{lm} j'_{l+n/2-1}(\gamma\l ) Y_l^m(\omega)\right).
$$
They match if and only if $a_{lm}$ and $b_{lm}$ solve the system
\be{system2} 
M_\nu \begin{pmatrix} a_{lm} \\ b_{lm}\end{pmatrix}:= 
\begin{pmatrix}
 j_\nu(\l)  &  j_\nu(\gamma\l)  \\ 
  j'_\nu(\l)  & \gamma j'_\nu(\gamma \l) 
 \end{pmatrix}  \begin{pmatrix} a_{lm} \\ b_{lm}\end{pmatrix} = 0.
\ee
Then $F_\nu=\det M_\nu$. If $F_\nu(\l)=0$ for some $\nu$ and $\l$, then \r{system2} has a nonzero solution for that $\nu$. Next, since $J_\nu$ and $J'_\nu$ cannot vanish simultaneously,  that solution consists of $c_m(a_{l},b_{l})$ with all coefficients non-zero; note that $A_\nu$ depends on $l$ but not on $m$. We may assume that $(a_l,b_l)$ is a  unit vector, generating the 1D null space.   Then we get non-zero solutions $u$ and $v$ with that fixed $\nu=l+n/2-1$ and any $m=1,\dots,\mu(l)$ of the form
\be{u1}
u_\nu(x) =   \sum_{m=1}^{\mu(l)} c_m a_{l} j_{\nu(l)}(\l r) Y_l^m(\omega), \quad
v_\nu(x) =   \sum_{m=1}^{\mu(l)} c_m b_{l} j_{\nu(l)}(\gamma l)(\gamma \l r) Y_l^m(\omega)
\ee
(recall that $\nu=l+n/2-1$). This gives a space of eigenpairs of dimension $\mu(l)$. If the same root $\l$ of $F_\nu$ happens to be  a root of one or more than one  $F_{\nu'}$ with $\nu'\not=\nu$, the corresponding eigenspaces are orthogonal, and the total dimension is the sum of the dimensions. Therefore, each such root contributes $\mu(l)$ ITEs to the counting function. 
In particular, the algebraic multiplicity of $\l$, defined as the order of $\l$ as a root of $F_\nu$, plays no role.

Denote by $\lambda_{\nu,j}$ the zeros of $F_\nu$ defined by \r{F}, with $\nu\in \mathbf{Z}+n/2-1$. Then the discussion above implies the following:
\be{Nr}
N(r) =  \sum_{\lambda_{\nu,j}<r}\mu(\nu-n/2+1)    .
\ee
Note that the sum is finite;  by \r{j_zeros} below, it contains zeros associated with $\nu=O(r)$ only. 

\subsection{Analysis of the zeros of $F_\nu$}
The Bessel functions $J_\nu(\l)$ solve
  $$ \l^2 J_{\nu}'' + \l J_{\nu}' + ( \l^2 - \nu^2) J_{\nu} = 0. $$
Thefefore,
  $$ J_{\nu}'' = - \l^{-1} J_{\nu}' - (1 - \nu^2 \l^{-2}) J_{\nu} $$
  $$J_{\nu}''(\gamma \l)  = - (\gamma \l)^{-1} J_{\nu}' (\gamma \l)- (1 - \nu^2 \gamma^{-2} \l^{-2}) J_{\nu} (\gamma \l).$$
 We drop the subscript $\nu$ in $F_\nu$ in the computations below.  We compute $F'(\l)$:
  \begin{align*}
  F'(\l ) = &\  \gamma J'_{\nu}(\l) J_{\nu}'(\gamma \l)
  + \gamma^2 J_{\nu}(\l) J_{\nu}''(\gamma \l) 
  - \gamma J'_{\nu}(\gamma \l) J_{\nu}'(\l) 
  - J_{\nu}(\gamma \l) J_{\nu}''(\l) \\
   = &\ \gamma^2 J_{\nu}(\l) J_{\nu}''(\gamma \l) - J_{\nu}(\gamma \l) J_{\nu}''(\l) \\
   = & \ -\gamma^2 J_{\nu}(\l) (\gamma \l)^{-1} J_{\nu}'(\gamma \l)
  -\gamma^2 J_{\nu}(\l) (1 - \nu^2 \gamma^{-2} \l^{-2}) J_{\nu}(\gamma \l)\\
 & \ + J_{\nu}(\gamma \l)  \l^{-1} J_{\nu}'(\l) + (1 - \nu^2 \l^{-2}) J_{\nu}(\gamma \l)J_{\nu}(\l)
   \\
   =&\  \l^{-1} \Big( -\gamma J_{\nu}(\l) J_{\nu}'(\gamma \l)+
   J_{\nu}(\gamma \l) J_{\nu}'(\l) \Big) 
    + (  1-\gamma^2) J_{\nu}(\l) J_{\nu}(\gamma \l) \\
   = & \ -\l^{-1} F(\l) + (  1-\gamma^2) J_{\nu}(\l) J_{\nu}(\gamma \l)  .
   \end{align*}
        The zeros of $J_{\nu}(\l) $ are at all simple 
    except possibly at $\l= 0 $. 
   Therefore, 
    $\l=\l_0\not=0$ is a zero of $F$ with multiplicity more than $1$ if and only if
    \begin{equation}\label{mul}
    \begin{cases} F(\l_0) = 0 \\
    F'(\l_0) = 0 \end{cases} \Leftrightarrow  \begin{cases} J_{\nu}(\l_0) = 0 \\ J_{\nu}(\gamma \l_0) = 0 \end{cases}.
    \end{equation}
    We compute $F''(\l)$ now: 
  $$F''(\l) = \l^{-2} F(\l) -\l^{-1} F'(\l)
  + (1-\gamma^2) J_{\nu}'(\l) J_{\nu}(\gamma \l)
  + (1-\gamma^2) \gamma J_{\nu}(\l) J'_{\nu}(\gamma \l)$$
  Hence at $\l_0\neq 0 $ satisfying \eqref{mul}, we have $F''(\l_0) = 0 $.
  
  Next, we compute $F'''(\l)$:
  \begin{align*}
  F'''(\l) = &\ -2 \l^{-3} F(\l) + \l^{-2}F'(\l) + \l^{-2} F'(\l) - \l^{-1} F''(\l) \\
  &\  + (1-\gamma^2)\left( J_{\nu}''(\l) J_{\nu}(\gamma \l)
  + 2 \gamma J_{\nu}'(\l) J'_{\nu}(\gamma \l) +\gamma^2 J_{\nu}(\l) J''_{\nu}(\gamma \l)\right).
  \end{align*}
  
  At $\l_0\neq 0 $ satisfying \eqref{mul}, we have  $F''(\l_0) = 0 $ and
  $$F'''(\l_0) =  2\gamma(1-\gamma^2) J_{\nu}'(\l) J_{\nu}'(\gamma \l).$$
    Since $J_{\nu}$ only has simple  roots away from $\l = 0$, 
     we have the following result: 

    \begin{prop}
   For $\gamma(1-\gamma^2)\not=0$, away from $\l=0$,
$$F_\nu(\l) =  \gamma J_{\nu}(\l) J_{\nu}'(\gamma \l) - J_{\nu}(\gamma \l) J_{\nu}'(\l) $$ 
has roots with possible multiplicities one or three. It is three if and only if the root is a common zero of $J_\nu(\l)$ and $J_\nu(\gamma\l)$, and in particular, real. 
  \end{prop}

    \subsection{Comparing the derivatives at the intersections --- Higher dimension}
  The zeros of $F_\nu$ are also the intersection points, i.e., the  zeros of the following equation
    \begin{equation}\label{TransEqn}
     \gamma J_{\nu}(\l) J_{\nu}'(\gamma \l) = J_{\nu}(\gamma \l) J_{\nu}'(\l).
     \end{equation}
 Near every simple zero, we have 
    $J_{\nu}(\l_0) , J_{\nu}(\gamma \l_0) \neq 0$. 
Then we can rewrite \eqref{TransEqn} near $\l_0$ as
$$ \gamma      \dfrac{J'_{\nu}(\gamma \l)}{J_{\nu}(\gamma \l)}=  \dfrac{J'_{\nu}(\l)}{J_{\nu}(\l)}. $$ 
We drop the subscripts $\nu$ in $F_\nu$ the next few lines again. 
Since
$$
P := \gamma  \dfrac{J'_{\nu}(\gamma \l)}{J_{\nu}(\gamma \l)}-  \dfrac{J'_{\nu}(\l)}{J_{\nu}(\l)}
= \frac{F}{J_{\nu}( \l)J_{\nu}(\gamma \l)} =(1-\gamma^2) \frac{F}{F'+\l^{-1}F},
$$
at any simple zero $\l_0$, we have
\[
P'(\l_0)= (1-\gamma^2)\frac{F'(F'+\l^{-1}F)- F( F'+\l^{-1}F )' }{ (F'+\l^{-1}F)^2}= 1-\gamma^2. 
\]

   So we have proved the following. 
   \begin{prop}\label{prop_2}
   For $0 < \gamma < 1$, at each intersection point (root) $\l_0$ of $H(\l) = \dfrac{J'_{\nu}(\l)}{J_{\nu}(\l)}$ and $G(\l) = \gamma H(\gamma \l)$, we have
   $$ H'(\l_0) < G'(\l_0).$$
When $\gamma>1$, we have
  $$ H'(\l_0) > G'(\l_0).$$
    \end{prop}

This proposition is our  main counting argument. 

\subsection{Counting the real ITE's and the  proof of Theorem~\ref{thm1} }   \label{sec_3.4}  
The zeros of $F_\nu$ are of two types: (1) points of intersection of the graphs of $H$ and $G$ away from their poles; and (2) common poles of $G$ and $H$ (common zeros of $J_\nu(\l)$ and $J_\nu(\gamma\l))$. Denote the positive zeroes of $J_\nu$ by $j_{\nu,k}$, $k=1,2,\dots$ (we suppress the dependence on $\nu$). Let $ 0<\gamma<1$ first. We  call below the intervals $(j_{\nu,k},j_{\nu,k+1}]$ between two consecutive zeros of $F_\nu$ ``small intervals''; and the intervals $(j_{\nu,k}/\gamma,j_{\nu,k+1}/\gamma]$ between two consecutive zeros of $J_\nu(\gamma\l)$ will be called ``large intervals''. At the endpoints of each small/large interval, the corresponding function $H$ or $G$, respectively, diverges to $\infty$ on the left; and to $-\infty$ to the right. If a branch of $H$ intersects a branch of $G$, this can happen at one point only, by  Proposition~\ref{prop_2}. We refer to Figure~\ref{pic2}, where $\gamma>1$.

\begin{figure}[!ht] 
  \centering
  \includegraphics[scale=0.9,keepaspectratio=true,trim=0cm 2cm 0cm 1.5cm, clip]{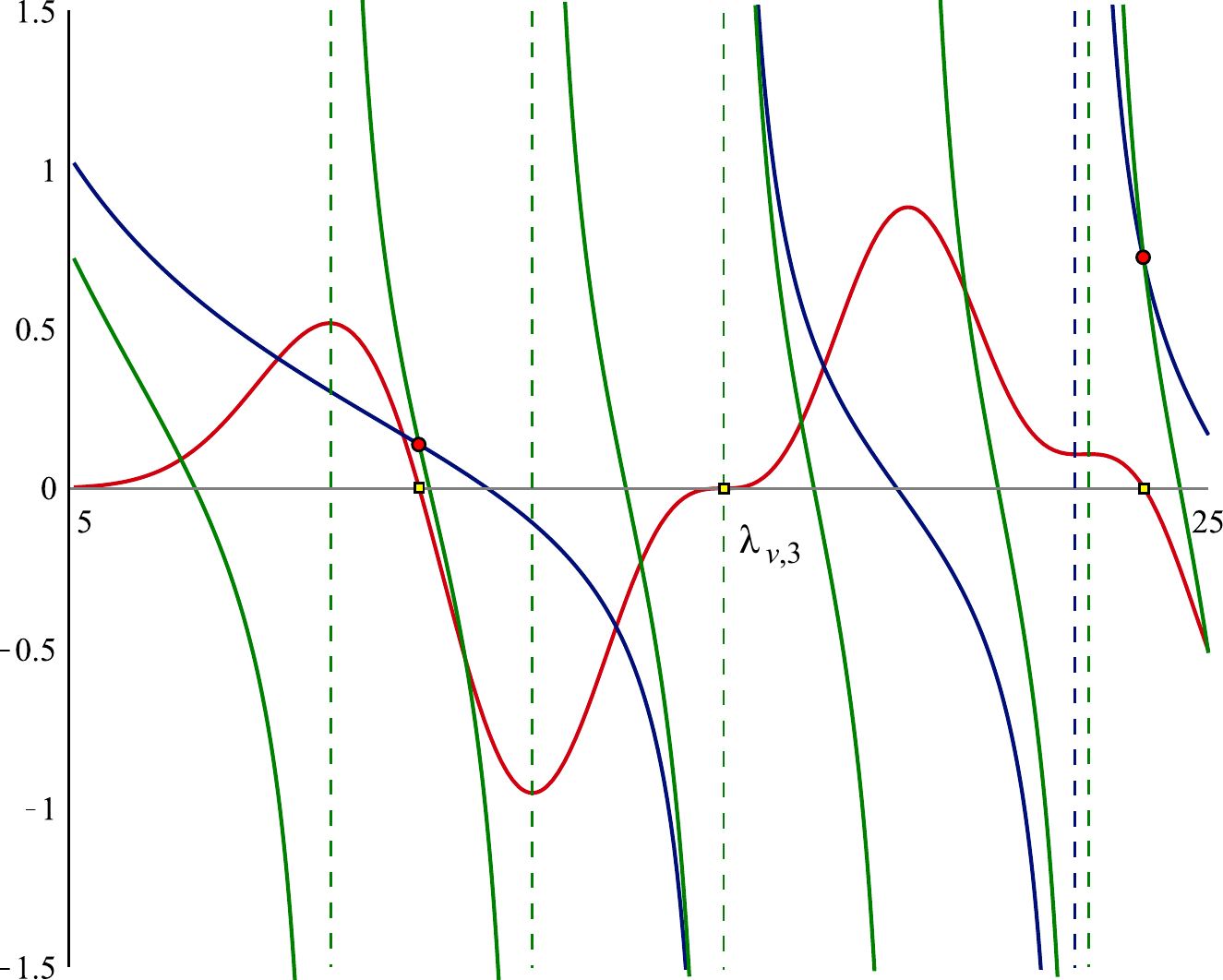}
\caption{The graphs of $F_\nu$ (the smooth curve), $H_\nu$ and $G_\nu$ for $\l\in [5,25]$ and $\nu=11/2$ with $\gamma= \lambda_{\nu,3}/\lambda_{\nu,1}\approx 0.57$. The function  $F_\nu$ has a triple root at $\lambda_{\nu,3}\approx 16.35 $ where the two vertical asymptotes coincide. The zeros to the left  and right   are simple. }
  \label{pic2}
\end{figure}

If a small interval is contained in the interior of  large one, then the graph of $J_\nu(\gamma\l)$ will intersect that of $J_\nu(\l)$, and there is exactly one such point in that small interval. The $\l$ coordinate of that point   is an $\nu$-ITE (a zero of $F_\nu$). If $J_\nu(\gamma\l)$ and $J_\nu(\l)$ have a common pole (vertical asymptote), then that pole is an $\nu$-ITE as well.   Those are the two types of $\nu$-ITEs we may have. In the latter case, a small interval is contained in the closure of a large one and they have a common endpoint. Since we defined all intervals as open on the  left and closed on the right; we may attribute an ITE which is a common pole to the small interval to the left of it. Therefore, we established an bijection between the $\nu$-ITEs and the small intervals which are contained entirely in a large one. The small intervals left without an associated $\nu$-ITE are those which have common points with two large ones; i.e., those containing some of the zeros $j_{\nu ,k}/\gamma$. Therefore, the number of $\nu$-ITEs not exceeding $r$ is equal to the number of zeros of $J_\nu(\l)$ minus that of $J_\nu(\gamma\l)$ up to an error 1, depending on the position on $r$ in the small interval to which it belongs. By \cite[Ch.~7.6.5]{Olver_book}, 
\be{j_zeros}
j_{\nu,k}= k\pi+   \frac12\nu\pi  -\frac14\pi+O(k^{-1}) .
\ee
Therefore, that error, multiplied by the corresponding multiplicity, see \r{Nr}, contributes an $O(r^{n-1})$ term to $N(r)$. This proves Theorem~\ref{thm1} for $0<\gamma<1$. The case $\gamma>1$ follows by rescaling, as in the 1D case. 

\subsection{About the multiplicities again}\label{sec_n_mult} The geometric multiplicity of each zero $\l_0$ of $F_\nu(l)$ is $\mu(l)$, if there is only one $\nu$ so that $F_\nu(\l_0)=0$; otherwise is a sum of such $\mu(\l)$. The ITE $\l_0$ is multiple (triple) if and only if $J_\nu(\l_0)=J_\nu(\gamma\l_0)=0$. Then we cannot tell whether $\gamma=1$ or not if the Cauchy data is $(0,Y_l)$   for any $Y_m$ a linear combination of $Y_l^m$, $m=1,\dots,\mu(l)$. However, we can do it for Cauchy data $(Y_l,0)$, which is the orthogonal complement to that space for a fixed $l$. Therefore, the algebraic multiplicity $3\mu(\l)$ does not play a role here. It only tells us how fast the information about $\gamma$ encoded in the Dirichlet data, ``degenerates'', as $\l\to\l_0$.

\section{Transmission eigenvalues (TEs)}\label{sec4}
It is easy to see that in this case, the interior transmission eigenvalues are also transmission eigenvalues (in the whole $\R^n$). Indeed, $u_\nu$ in the eigenpair in \r{u1} extends from the unit ball to the whole $\R^n$ in a trivial way, by the same formula. Then the function $v_\nu$, extended as $u_\nu$ outside the unit ball as $u_\nu$, is a solution of $(-\Delta-\l^2 m)u=0$, where $m=\gamma^2$ in the unit ball, and $m=1$ outside. This is a transmission problem. We will use the following facts:  $u_\nu$ is $C^\infty$, and its the exterior Cauchy data matches the interior one;  the interior one is the same as that of $v_\nu$ because $(u_\nu,v_\nu)$ is an eigenpair; the exterior Cauchy data of both functions coincide as well because they are equal outside the unit ball. Therefore, $v_\nu$ and its normal derivative do not jump across the unit sphere. The relative scattering matrix in the spherical harmonic base was computed in \cite{St-sharp_bd}. It is a diagonal operator with diagonal entries
$$
S_l(\l) = -\frac
{  h_\nu^{(2)'}(\l )j_\nu(\gamma\l) -\gamma h_\nu^{(2)}(\l )j_\nu'(\gamma\l) }
{  h_\nu^{(1)'}(\l )j_\nu(\gamma\l) - \gamma h_\nu^{(1)}(\l )j_\nu'(\gamma \l) },
$$
where $h_\nu^{(1,2)}(\l) = \l^{1-n/2}H_{l+n/2-1}^{(1,2)}(\l)$. 

Then $A_l(\l) = S_l(\l)-1$ are the diagonal elements of the scattering amplitude $A(\l)$, considered as an operator. A simple calculation yields
$$
A_l(\l) = \frac
{ - j_\nu'(\l )j_\nu(\gamma\l) +\gamma j_\nu(\l )j_\nu'(\gamma\l) }
{  h_\nu^{(1)'}(\l )j_\nu(\gamma\l) - \gamma h_\nu^{(1)}(\l )j_\nu'(\gamma \l) }
= \frac
{ F_\nu(\l)}
{  h_\nu^{(1)'}(\l )j_\nu(\gamma\l) - \gamma h_\nu^{(1)}(\l )j_\nu'(\gamma \l) }.
$$
Therefore, $A_l$ has the same zeros as $F_\nu$. The denominator has complex zeros only, at the resonances; which lie in the lower half-plane, see \cite{St-sharp_bd}. At each such zero, the eigenspace of $S(\l)$ restricted to the spherical harmonics with momentum $l$, has a kernel coinciding with that space; and its dimension is $\mu(l)$. Then we can define the geometric multiplicity of each TE $\l$ as the dimension of the kernel of $A(\l)$ (the latter my include more than one but always finitely many $l$'s). Now, $N(\l)$ is the number of the real $\l\not=0$ for which $A(\l)$ has a non-trivial kernel, counted with their geometric multiplicities. When an ITE $\l_0$ is not a simple toot of $F_\nu$, the algebraic multiplicity shows up if we study $A^{-1}(\l)$ --- the most singular term in the Laurent expansion is $(\l-\l_0)^{-3}$. The residue however,  and the whole singular part cannot have rank  greater than the dimension $\mu(l)$ of the spherical harmonics, which is the geometric multiplicity. 

\section{ITEs are not always TEs} 
We present here a simple example showing that ITEs are not always TEs (the converse is  clearly true). Take any solution $u$ of the Helmholtz equation $(-\Delta-\l^2)u=0$ for some $\l>0$ in $\Omega$ with the following properties: $u>0$ in $\bar\Omega$, $u$ is $C^\infty$ outside $\Omega$ but has no extension as a solution in the whole $\R^n$. Such a solution $u$, $\l$ and $\Omega$ are easy to construct; for example, if $n=3$, fix $\l>0$ and take $u= \cos(\l x)/|x|$ (the real  part of the Green's function, up to a constant), and $\Omega$ can be any domain in the ball $B(0,\pi/\l)$ so that $0\not\in\bar\Omega$. Now, take $\phi \in C_0^\infty(\bar\Omega;\;\R)$, and set $v=u+\epsilon \phi$. Then $v$ solves 
\[
(-\Delta-\l^2m)v=0\quad \text{in $\Omega$}\quad \text{with $m := -\frac{\Delta (u+\epsilon \phi)}{\l^2(u+\epsilon\phi)} =    \frac{u -\epsilon\l^{-2}\Delta \phi}{u+\epsilon\phi}         $}. 
\]
When $|\epsilon|\ll1$, $m$ is a well defined positive function in $\bar\Omega$ and $\l$ is an ITE because $u$ and $v$ have the same Cauchy data on $\partial\Omega$. On the other hand, $\l$ is  not a TE because $u$ does not extend as a solution in the whole $\R^n$.  If $(-\Delta-\l^2)\phi\not\equiv 0$, then $m\not\equiv1$.

\bibliographystyle{siam}
\bibliography{BiblioITE}

\end{document}